\documentclass{article}
\usepackage{amsmath}
\usepackage{amsfonts}
\usepackage{subfig}
\usepackage{graphicx}
\usepackage{amssymb}
\usepackage{authblk}
\usepackage{comment}
\usepackage[mathscr]{euscript}
\usepackage{float}
\usepackage[utf8]{inputenc}
\usepackage{multirow}
\usepackage{multicol}
\usepackage{array}
\usepackage[table,xcdraw]{xcolor}
\usepackage{lineno}
\usepackage{ifpdf}
\usepackage{latexsym}
\usepackage{algorithmic}
\usepackage{textcomp}
\usepackage{comment}
\usepackage{xcolor, colortbl}

\usepackage{hyperref}

\newcommand{\NN}{\mathbb{N}}

\newcommand{\RR}{\mathbb{R}}
\newcommand{\QQ}{\mathbb{Q}}
\newcommand{\ZZ}{\mathbb{Z}}

\renewcommand{\a}{\alpha}

\newcommand{\e}{\varepsilon}
\renewcommand{\d}{\delta}

\newcommand{\G}{\Gamma}

\renewcommand{\O}{\Omega}
\renewcommand{\o}{\omega}

\newtheorem{theorem}{Theorem}
\newtheorem{proof}{Proof}

\newtheorem{corollary}[theorem]{Corollary}

\newtheorem{definition}[theorem]{Definition}

\newtheorem{lemma}[theorem]{Lemma}

\newtheorem{proposition}[theorem]{Proposition}
\newtheorem{remark}[theorem]{Remark}

\title{Generalized Bernoulli Differential Equation}
 \author[1]{Hector Carmenate}

\author[2]{Paul Bosch}

\author[3]{Juan E. Nápoles}

\author[4]{José M. Sigarreta}

  \affil[1]{Mathematics and Natural Sciences Department, Padron Campus, Miami Dade College, 627 SW 27th Ave, Miami, 33135, Fl, United States; E-mail: hcarmena@mdc.edu}

  \affil[2]{Facultad de Ingeniería, Universidad del Desarrollo, Ave. La Plaza 680, San Carlos de Apoquindo, Las Condes, 7550000, Santiago, Chile; E-mail: pbosch@udd.cl}

  \affil[3]{Facultad de Ciencias Exactas y Naturales
y Agrimensura, Universidad Nacional del Nordeste, Corrientes 3400, Argentina; E-mail: jnapoles@exa.unne.edu.ar}
  
  \affil[4]{Facultad de Matemáticas, Universidad 
Autónoma de Guerrero, Carlos E. Adame No.54 Col. 
Garita, Acapulco Gro. 39650, Mexico; E-mail: 14366@uagro.mx}
  
\date{January 2023}

\begin{document}

\maketitle

\begin{abstract}
{In this paper we propose and solve a generalization of the Bernoulli Differential Equation, by means of a generalized fractional derivative. First we prove a generalization of Gronwall's inequality, which is useful for studying the stability of systems of fractional differential equations and we state results about the qualitative behavior of the trivial solution of the proposed equation. After that, we prove and state the main results about the solution of the generalized Bernoulli Differential Equation and also we give some examples that show the advantage of considering this
fractional derivative approach. We also present a finite difference method as an alternative to the solution of the generalized Bernoulli equation and prove its validity by means of examples.}
\end{abstract}
keywords: generalized local derivative, boundedness of solutions, Gronwall inequality\\

\section{Introduction}
Fractional calculus is a branch of mathematics that studies the differentiation and integration operators of generalized orders, fractional calculus had its birth almost at the same time as traditional calculus. But its application has already been widely demonstrated, for some applications refer to the following works  \cite{Khalil3, Baleanu04, Baleanu06}. Fractional Calculus has had different difficulties in the study of systems of fractional differential equations, although some approaches can be found in \cite{FMNS,G2017,PABB}.

The global fractional derivatives collect information on an interval and keep track of the history of the process, can be said to possess a certain memory; which makes it possible to model non-local and distributed responses that commonly appear in natural and physical phenomena, although it is known that they have certain limitations.  \cite{Khalil} defines a conformable fractional derivative that eliminates these disadvantages; more recently, a non-conformable local fractional derivative is introduced  in \cite{GLLMN}. We can consider fractional local derivatives to be a new tool, which has been proven its usefulness in multiple applications for several autors, see for example \cite {Cortez, FGNRS, GLLMN, Khalil3}.

\smallskip

This paper relies on use of new differential operators, depending on a general kernel function $T(t,\alpha)$, which include at once several local derivatives earlier introduced and studied in many different sources. This new tool is very powerful because we can model a phenomenon from two points of view: Considering different kernels and varying the order involved in said kernel.
\smallskip

It's know that one of the most paradigmatic nonlinear equations is the Bernoulli differential equation (which is know was introduced by Jacob in his work \cite{JB}), this equation can be considered as a statement of the principle of conservation of energy in fluids. In this paper we'll study Bernoulli's equation under the generalized fractional local derivative operator, but first we prove a generalized Gronwell's inequality and we obtain stability conditions for systems of differential fractional  equations under the approach of the generalized local fractional derivative, after that we find solvability and stability conditions for the generalized Bernoulli's equation proposed and finally, we give some examples showing particular cases of this Bernoulli's equation seen from different approaches.

There are cases in which the fractional differential equation becomes too complicated to solve by classical methods and other alternatives must be analyzed, numerical methods are a very useful tool when we must solve this type of equations. The finite difference methods replace the derivative operator by an appropriate quotient in differences that allows to approximate the solution quite efficiently. Here we use a finite difference method to solve the generalized Bernoulli equation and show the results obtained by means of examples.

\section{Generalized Derivation and Integration Operators}
In this section we present a  definition of a generalized local fractional derivative introduced  in \cite{FNRS} as well as a fractional integration operator introduced in \cite{FGNRS} and some of its most important properties that will be useful in the next section.\\

We will now present some advantages of studying the Bernoulli equation using the generalized differentiation operator.

The derivative we are considering generalizes many of the properties of the local derivatives existing so far, it also allows the computation of higher order derivatives and is not limited only to functions defined on the positive half-line. It is important to emphasize that the choice of the kernel $T(t,\alpha )$ leads to different practical applications. Thanks to the generality of the theoretical results obtained, we can state that they do not depend on the choice of the kernel. In the same direction, different applications of the generalized fractional derivatives are shown, as well as their relations with other types of local  derivatives(conformable or not), for example \cite{Flei}:
\begin{enumerate}
    \item If $\alpha \in (0,1]$ and $T(t,\alpha )=t^{1-\alpha }$, then the conformable fractional derivative defined in \cite{Khalil} is obtained.
    \item If $\a \in (0,1]$ and $T(t,\alpha )=k(t)^{1-\alpha }$, then the general conformable fractional derivative defined in \cite{Almeida} is derived.
    \item If $\a \in (0,1]$ and $T(t,\alpha )=e^{t^{-\alpha }}$, then the non-conformable fractional derivative defined in \cite{GLLMN} is obtained.
    \item If $\a \in (0,1]$ and $T(t,\alpha )={ \left( t+\frac { 1 }{ \Gamma (\alpha ) }  \right)  }^{ 1-\alpha  }$, then we obtain the beta-derivative defined in \cite{Atan}.
\end{enumerate}

In \cite{FNRS} the definition of generalized fractional derivative is given as follows.

Given $s \in \RR$, we denote by $\lceil s \rceil$ the \emph{upper integer part} of $s$, i.e., the smallest integer
greater than or equal to $s$.

\begin{definition}
\label{d:g0}
Given an interval $I \subseteq \RR$, $f: I \rightarrow {\mathbb{R}}$, $\alpha \in \RR^+$ and a positive
continuous function $T(t,\alpha )$ on $I$,
the \emph{derivative} $G_{T}^{\alpha }f$ of $f$ of order $\alpha$ at the point $t \in I$ is defined by
\begin{align}
G_{T}^{\alpha }f(t) & =\lim_{h\to 0} \frac{1}{h^{\lceil \a \rceil}}\nonumber\\
& \times \sum_{k=0}^{\lceil \a \rceil}(-1)^{k}\binom{\lceil \a \rceil}{k}f\big(t-khT(t,\alpha )\big) .
\end{align}
\end{definition}

If $a=\min\{t \in I\}$ (respectively, $b=\max\{t \in I\}$), then $G_{T}^{\alpha }f(a)$ (respectively, $G_{T}^{\alpha }f(b)$) is defined with
$h\to 0^-$ (respectively, $h\to 0^+$) instead of $h\to 0$ in the limit.\\

If we choose the function $T(t,\alpha )= t^{\lceil \a \rceil-\alpha }$, then we obtain the following particular case of the function $G_{T}^{\alpha }$ which is a conformable derivative.

\begin{definition}
\label{d:g}
Let $I$ be an interval $I \subseteq (0,\infty)$, $f: I \rightarrow {\mathbb{R}}$ and $\alpha \in \RR^+$. The
\emph{conformable derivative} $%
G^\a f$ of $f$ of order $\alpha$ at the point $t \in I$ is defined by
\begin{align}
G^{\alpha }f(t)
& =\lim_{h\rightarrow 0} \frac{1}{h^{\lceil \a \rceil}} \nonumber\\
& \times \sum_{k=0}^{\lceil \a \rceil}(-1)^{k}\binom{\lceil \a \rceil}{k}
f\big(t-kht^{\lceil \a \rceil-\alpha }\big).
\end{align}
\end{definition}

Note that, if $\alpha =n \in \NN$ and $f$ is smooth enough, then Definition \ref{d:g} coincides with the classical definition of the $n$-th derivative.\\

In \cite{Khalil} is defined a conformable derivative $T_{\alpha }$ which is a particular case of $G^{\alpha }$ when $\alpha \in (0,1]$ and $T(t,\alpha )=
t^{1-\alpha }$. See \cite{Abdejjawad}, \cite{Jarad} and \cite{Katugampola} for more information on $T_{\alpha }$.

\medskip

The following results in \cite{FNRS} contain some basic properties of the derivative $G_{T}^{\a}$.
\begin{theorem}
\label{t:comp}
Let $I$ be an interval $I \subseteq \RR$, $f: I \rightarrow {\mathbb{R}}$ and $\alpha \in \RR^+$.

$(1)$ If there exists $D^{\lceil \a \rceil}f$ at the point $t \in I$, then $f$ is $G_{T}^\a$-differentiable at $t$ and
$G_{T}^{\alpha }f(t)= T(t,\alpha )^{\lceil \a \rceil}D^{\lceil \a \rceil}f(t)$.

$(2)$ If $\a \in (0,1]$, then $f$ is $G_{T}^\a$-differentiable at $t \in I$ if and only if $f$ is differentiable at
$t$;
in this case, we have $G_{T}^{\alpha }f(t)= T(t,\alpha ) f'(t)$.
\end{theorem}

\begin{theorem}
\label{t:prop}
Let $I$ be an interval $I \subseteq \RR$, $f,g: I \rightarrow {\mathbb{R}}$ and $\alpha \in \RR^+$.
Assume that $f,g$ are $G_{T}^\a$-differentiable functions at $t \in I$.
Then the following statements hold:

$(1)$ $af+b\,g$ is $G_{T}^\a$-differentiable at $t$ for every $a,b \in \mathbb{R}$, and
$G_{T}^{\alpha }(af+b\,g)(t) = a\,G_{T}^{\alpha }f(t)+b\,G_{T}^{\alpha }g(t)$.

$(2)$ If $\alpha \in (0,1]$, then $fg$ is $G_{T}^\a$-differentiable at $t$ and
$G_{T}^{\alpha }(fg)(t)=f(t)G_{T}^{\alpha }g(t)+g(t)G_{T}^{\alpha }f(t)$.

$(3)$ If $\alpha \in (0,1]$ and $g(t) \neq 0$, then $f/g$ is $G_{T}^\a$-differentiable at $t$ and
$G_{T}^{\alpha }(\frac{f}{g})(t)=\frac{g(t)G_{T}^{\alpha }f(t)-f(t)G_{T}^{\alpha }g(t)}{g(t)^{2}}$.

$(4)$ $G_{T}^{\alpha }(\lambda )=0$, for every $\lambda \in \mathbb{R}.$

$(5)$ $G_{T}^{\alpha }(t^{p}) = \frac{\G(p+1)}{\G(p-\lceil\alpha \rceil+1)}t^{p-\lceil\alpha \rceil} T(t,\alpha
)^{{\lceil\alpha \rceil}}$, for every $p\in \mathbb{R} \setminus \ZZ^-\!.$

$(6)$ $G_{T}^{\alpha }(t^{-n}) = (-1)^{\lceil\alpha \rceil}\frac{\G(n+\lceil\alpha \rceil)}{\G(n)}\, t^{-n-\lceil\alpha
\rceil} T(t,\alpha )^{{\lceil\alpha \rceil}}$, for every $n\in \mathbb{Z}^+.$
\end{theorem}

\begin{theorem}
\label{t:chain}
Let $\alpha \in (0,1]$, $g$ a $G_{T}^\a$-differentiable function at $t$ and $f$ a differentiable function at $g(t)$.
Then $f\circ g$ is $G_{T}^{\alpha }$-differentiable at $t$, and
$G_{T}^{\alpha }(f\circ g)(t)=f'(g(t))\,G_{T}^{\alpha }g(t)$.
\end{theorem}
In \cite{FGNRS} is defined an integral operator in the following way.

Let $I$ be an interval $I \subseteq \RR$, $a,t \in I$ and $\a \in \RR$.
The integral operator $J_{T,a}^\a$ is defined for every locally integrable function $f$ on $I$ as
$$
J_{T,a}^\a(f)(t)=\int _{ a }^{ t } \frac { f(\o) }{ T(\o,\alpha ) } \; d\o.
$$

The following results appear in \cite{FGNRS}.

\begin{proposition}
\label{p:fundamental}
Let $I$ be an interval $I \subseteq \RR$, $a \in I$, $0<\alpha \le 1$ and $f$ a differentiable function on $I$ such that $f'$ is a locally integrable function on $I$.
Then, we have for all $t \in I$
$$
J_{T,a}^\a \big({ G }_{ T }^{ \alpha  }(f)\big)(t)=f(t)-f(a).
$$
\end{proposition}

\begin{proposition}
\label{p:inverse}
Let $I$ be an interval $I \subseteq \RR$, $a \in I$ and $\a \in (0,1]$.
$$
G_T^{\a} \big( J_{T,a}^\a(f)\big) (t)
= f(t),
$$
for every continuous function $f$ on $I$ and $a,t\in I$.
\end{proposition}

\medskip

In \cite{Khalil} it is defined the integral operator $J_{T,a}^\a$ with $T$ given by $T(t,\a)= t^{1-\a}$,
and \cite[Theorem 3.1]{Khalil} shows
$$
G^{\a} J_{t^{1-\a}\!,\,a}^\a(f)(t)
= f(t),
$$
for every continuous function $f$ on $I$, $a,t \in I$ and $\a \in (0,1]$.
Hence, Proposition \ref{p:inverse} extends to any $T$ this important equality.

\medskip

The following result summarizes some elementary properties of the integral operator $J_{T,a}^\a$.

\begin{theorem}
\label{t:pos}
Let $I$ be an interval $I \subseteq \RR$, $a,b \in I$ and $\a \in \RR$.
Suppose that $f,g$ are locally integrable functions on $I$, and $k_1,k_2\in \RR$. Then we have

$(1)$ $J_{T,a}^\a \big( k_1 f+k_2 g \big) (t) =k_1 J_{T,a}^\a f(t) + k_2 J_{T,a}^\a g(t),$

$(2)$ if $f \ge g$, then $J_{T,a}^\a f(t)\ge J_{T,a}^\a g(t)$ for every $t\in I$ with $t \ge a$,

$(3)$ $\left| J_{T,a}^\a f(t) \right| \le J_{T,a}^\a \left| f \right| (t)$ for every $t\in I$ with $t \ge a$,

$(4)$ $\int _{ a }^{ b } \frac { f(\o) }{ T(\o,\alpha ) } d\o
=J_{ T,a }^{ \alpha  }f(t)-J_{ T,b }^{ \alpha  }f(t)
=J_{ T,a }^{ \alpha  }f(t)(b)$ for every $t \in I$.
\end{theorem}

\begin{remark}
The above results generalize Proposition 1, Proposition 2 and Theorem 1 of \cite {GLNV}, respectively, obtained with $ 0<\alpha \le 1$.
\end{remark}

\medskip

Given $a,b \in I$ $(b>a)$, let us denote by $F_{a,b}$ the usual inner product in $L^2[a,b]$
$$
F_{a,b} (f,g) = \int_a^b f(t) g(t) \,dt .
$$

\begin{proposition}
\label{p:adjoint}
Let $I$ be an interval $I \subseteq \RR$, $a,b \in I$ with $a<b$ and $\a \in \RR$.
The adjoint of $J_{T,a}^\a$ in $L^2[a,b]$ with respect to the inner product $F_{a,b}$ is the operator
$$
A^\a_{T,a,b}(f)(t)
= \frac{1}{T(t,\a)}\int_t^b f(s)\, ds .
$$
\end{proposition}

\begin{proof}
We obtain, by using integration by parts,
$$
\begin{aligned}
&F_{a,b} \big( J_{T,a}^\a(f) ,g \big)
 = \int_a^b g(t) J_{T,a}^\a(f)(t)\,dt
\\
& = \Big[ -J_{T,a}^\a(f)(t) \int_t^b g(s)\, ds \Big]_{t=a}^{t=b} \\
& + \int_a^b \frac{f(t)}{T(t,\a)}\int_t^b g(s)\, ds\,dt
\\
& = \int_a^b \frac{f(t)}{T(t,\a)}\int_t^b g(s)\, ds\,dt
= \int_a^b f(t)\,A^\a_{T,a,b}(g)(t)\,dt
\\
& = F_{a,b} \big( f ,A^\a_{T,a,b}(g) \big).
\end{aligned}
$$
\end{proof}

\begin{proposition}
\label{p:HS}
Let $I$ be an interval $I \subseteq \RR$, $a,b \in I$ with $a<b$ and $\a \in \RR$.
Then $J_{T,a}^\a$ is a Hilbert-Schmidt integral operator on $L^2[a,b]$, and so, a continuous and compact operator.
\end{proposition}

\begin{proof}
Let us denote by $\chi_{_A}$ the \emph{characteristic function} of the set $A$ (i.e., $\chi_{_A}(t)=1$ if $t \in A$ and $\chi_{_A}(t)=0$ otherwise).
Then
$$
J_{T,a}^\a(f)(t)
= \int_a^t \frac{f(\o)}{T(\o,\a)}\; d\o
= \int_a^b k(t,\o) f(\o) \, d\o ,
$$
with
$$
k(t,\o)
= \chi_{_{[a,t]}}(\o)\, \frac{1}{T(\o,\a)} \,.
$$
We have
$$
\begin{aligned}
\int_a^b \int_a^b |k(t,\o)|^2 \, d\o\,dt
& \le \int_a^b \int_a^b \frac{1}{T(\o,\a)^2} \, d\o\,dt\\
& = \int_a^b \frac{b-a}{T(\o,\a)^2} \, d\o < \infty,
\end{aligned}
$$
since $T(\o,\a)$ is a positive continuous function on $I$.
Thus, $J_{T,a}^\a$ is a Hilbert-Schmidt integral operator on $L^2[a,b]$, and so, a continuous and compact operator.
\end{proof}

\medskip
The integral operator $J_{T,a}^\a$ can be applied in order to solve fractional differential equations. Let $p$ and $q$ be continuous functions on an interval $I \subseteq \RR$, $a \in I$ and $\a \in (0,1]$.
Then the solution of the linear fractional equation
\begin{equation}
\label{eq:lineargen}
G_{ T }^{ \alpha  }y+p(t)y=q(t),
\end{equation}
on $I$ is given by
\begin{equation}
\label{eq:solgen}
y(t)={ e }^{ -J_{T,a}^\a (p)(t) }\left( J_{T,a}^\a \big( q { e }^{ J_{T,a}^\a(p) }\big)(t) +C \right) .
\end{equation}

\subsection{Generalized Bernoulli Differential Equation}

Next, we prove a fractional version of Gronwall's inequality which will be useful in the study of the stability of systems of fractional differential equations. A version of this inequality was proved in \cite{Abdejjawad}.

\begin{theorem}
\label{t:gronwall}
Let $r$, $c$, $d$ and $k$ be continuous functions on the interval $[a, b]$, $d,k \ge 0$, and $\alpha \in (0,1]$, such that
$$
r(t)\le c(t)+d(t) J_{T,a}^\a(kr)(t),
$$
for all $t\in [a,b]$.
Then we have
$$
r(t)\le c(t)+d(t)J_{ T,a }^{ \alpha  }(cke^{ -J_{ T,a }^{ \alpha  }(dk) })(t) e^{ J_{ T,a }^\alpha (dk)(t)} ,
$$
for every $t\in [a,b]$.
\end{theorem}

\begin{proof}
Let us define $R(t)= J_{T,a}^\a(kr)(t)$.
Thus, $R(a)=0$,
$$
r(t)\le c(t)+d(t) R(t)
$$
for all $t\in [a,b]$,
and Proposition \ref{p:inverse} gives
\begin{equation}
\label{eq:gronwall}
G_{ T }^{ \alpha  }R(t)
= k(t)r(t)
\le c(t)k(t) + d(t)k(t) R(t),
\end{equation}
for every $t\in [a,b]$, since $k\ge 0$.
Let us define $E(t)= e^{ -J_{T,a}^\a (dk)(t)}$.
Theorem \ref{t:chain} and Proposition \ref{p:inverse} give
$$
\begin{aligned}
G_{ T }^{ \alpha  } E(t)
& = G_{ T }^{ \alpha  } \big( e^{ -J_{T,a}^\a (dk)} \big)(t)\\
& = G_{ T }^{ \alpha  } \big( -J_{T,a}^\a (dk) \big)(t) e^{ -J_{T,a}^\a (dk)(t)}\\
& = -d(t)k(t) E(t) .
\end{aligned}
$$
Since $R$ and $E$ are $G_T^\a$-differentiable on $[a,b]$,
Theorem \ref{t:prop} and \eqref{eq:gronwall} give
$$
\begin{aligned}
G_{ T }^{ \alpha  }( E R )(t)
& = E(t) G_{ T }^{ \alpha  } R(t) + R(t) G_{ T }^{ \alpha  } E(t)\\
& = E(t) G_{ T }^{ \alpha  } R(t) -d(t)k(t) R(t) E(t)
\\
& \le c(t)k(t)E(t) + d(t)k(t) R(t)E(t)\\
& -d(t)k(t) R(t) E(t) = c(t)k(t)E(t) ,
\end{aligned}
$$
for every $t \in [a,b]$, since $E \ge 0$.
Since $E(t)R(t)$ is differentiable on $[a, b]$ by Theorem \ref{t:comp},
we have that Theorem \ref{t:pos} and Proposition \ref{p:fundamental} give
$$
\begin{aligned}
J_{T,a}^\a \big( ck E \big)(t)
& \ge J_{T,a}^\a \big( G_{ T }^{ \alpha  }( ER) \big)(t)\\
& =E(t)R(t)-E(a)R(a)
=E(t)R(t),
\end{aligned}
$$
for any $t\in [a,b]$.
Thus, since $d, E \ge 0$,
$$
\begin{aligned}
r(t)
& \le c(t)+d(t) R(t)
\le c(t)+d(t) \frac{J_{T,a}^\a \big( ck E \big)(t)}{E(t)}
\\
& = c(t)+d(t) J_{T,a}^\a \big( ck e^{-J_{T,a}^\a(dk)} \big)(t)\, e^{J_{T,a}^\a(dk)(t)},
\end{aligned}
$$
for every $t\in [a,b]$.
\end{proof}

\medskip

The argument in the proof of Theorem \ref{t:gronwall} also gives the following converse inequality.

\begin{theorem}
\label{t:gronwallconverse}
Let $r$, $c$, $d$ and $k$ be continuous functions on the interval $[a, b]$, $d,k \ge 0$, and $\alpha \in (0,1]$, such that
$$
r(t)\ge c(t)+d(t) J_{T,a}^\a(kr)(t),
$$
for all $t\in [a,b]$.
Then we have
$$
r(t)\ge c(t)+d(t)J_{ T,a }^\alpha (cke^{ -J_{ T,a }^\alpha (dk) })(t) e^{ J_{ T,a }^{ \alpha  }(dk)(t)} ,
$$
for every $t\in [a,b]$.
\end{theorem}

\medskip

If $c\ge 0$, then the conclusion of Theorem \ref{t:gronwall} can be simplified.

\begin{theorem}
Let $r$, $c$, $d$ and $k$ be continuous functions on the interval $[a, b]$, $c,d,k \ge 0$, and $\alpha \in (0,1]$, such that
$$
r(t)\le c(t)+d(t) J_{T,a}^\a(kr)(t),
$$
for all $t\in [a,b]$.
Then we have
$$
r(t)\le c(t) + d(t) J_{T,a}^\a ( ck )(t) \, e^{ J_{T,a}^\a(dk)(t) },
$$
for every $t\in [a,b]$.
\end{theorem}

\begin{proof}
Theorem \ref{t:gronwall} gives
$$
r(t)\le c(t) + d(t) J_{T,a}^\a\big( cke^{ -J_{T,a}^\a(dk) } \big)(t) \, e^{ J_{T,a}^\a(dk)(t) },
$$
for every $t\in [a,b]$.
Since $d \ge 0$, it suffices to prove that
\begin{equation}
\label{eq:gr}
J_{T,a}^\a\big( cke^{ -J_{T,a}^\a(dk) } \big)(t)
\le J_{T,a}^\a ( ck )(t)
\end{equation}
for any $t\in [a,b]$.

Since $dk \ge 0$, Theorem \ref{t:pos} gives $J_{T,a}^\a(dk) \ge 0$ and so, $e^{ -J_{T,a}^\a(dk)} \le 1$.
Since $ck \ge 0$, we have $cke^{ -J_{T,a}^\a(dk)} \le ck$,
and Theorem \ref{t:pos} gives \eqref{eq:gr}.
\end{proof}

The following result appears in \cite{FGNRS}.

\begin{lemma}
\label{t:Picard}
Let $\a_1,\a_2,\dots,\a_n \in (0,1]$, $x=(x_1, \dots, x_n)$,
$I \subseteq \RR$ an interval, $\O \subseteq \RR^n$ an open set,
$t_0 \in I$ and $x_0 \in \O$.
Let $F=(F_1,\dots,F_n): I\times \O \rightarrow \RR^n$ be in $(C, Lip)$ on some open neighborhood of the point $(t_0, x_0)$, 
and consider the initial value problem
\begin{equation}
\label{eq:Picard}
G_T^{\a_j} x_j = F_j(t,x), \quad  1\le j \le n,  \qquad x(t_0) = x_0.
\end{equation}
Then there exists $h > 0$ such that \eqref{eq:Picard} has a unique solution on the interval $[t_0 - h, t_0 + h] \cap I$. Furthermore, if $\O = \RR^n$ and $F$ is in $(C, Lip)$ on $J \times \RR^n$ for each compact interval $J \subseteq I$, then \eqref{eq:Picard} has a unique solution on $I$.
\end{lemma}

Let $\a \in (0,1]$, $a \in \RR$, $t_0 \ge a$, $x_0 \in \RR^n$,
and $F: [a,\infty) \times \RR^n \rightarrow \RR^n$ in $(C, Lip)$ on $[a,\infty) \times \RR^n$.
Let us consider the initial value problem
\begin{equation}
\label{eq:Cauchy}
{ G }_{ T }^{ \alpha  }x(t)=F(t,x),\quad \quad
x(t_0)={ x }_{ 0 },
\end{equation}
Lemma \ref{t:Picard} guarantees that \eqref{eq:Cauchy} has a unique solution on $[a,\infty)$.
The study of boundedness of solutions of a differential equation, either fractional or not, plays an important role in qualitative theory.
In addition, the qualitative behavior of solutions plays an important role in many real-world phenomena related to applied research.
Based on the previous results, we can obtain stability results for the solutions of fractional differential equations.

\begin{definition}
If $F(t,0)=0$ for every $t \ge a$, then the trivial solution $x\equiv 0$ of \eqref{eq:Cauchy} is said to be
\emph{stable} if for any $\varepsilon >0$, there exists $\delta =\delta (t_0,\varepsilon )>0$
such that if $\left| x_{ 0 } \right| <\delta $, then $\left| x(t) \right| <\varepsilon$ for every $t \ge t_0$;
it is \emph{uniformly stable} if there exists $\delta =\delta (\varepsilon )>0$
such that if $\left| x_{ 0 } \right| <\delta $, then $\left| x(t) \right| <\varepsilon$ for every $t \ge t_0 \ge a$.
\end{definition}

\smallskip

Propositions \ref{p:fundamental} and \ref{p:inverse} give the following result.

\begin{proposition}
\label{p:equivalence}
Let $\a \in (0,1]$, $a \in \RR$, $t_0 \ge a$, $x_0 \in \RR^n$,
and $F: [a,\infty) \times \RR^n \rightarrow \RR^n$ in $(C, Lip)$ on $[a,\infty) \times \RR^n$.
Then the problem \eqref{eq:Cauchy} is equivalent to
\begin{equation}
\label{eq:eqintegral}
x(t)={ x }_{ 0 }+J_{ T,t_0 }^{\a }F(s,x(s))(t).
\end{equation}
\end{proposition}

\begin{theorem}
\label{t:boundedness1}
Let $\a \in (0,1]$, $a \in \RR$
and $F: [a,\infty) \times \RR^n \rightarrow \RR^n$ in $(C, Lip)$ on $[a,\infty) \times \RR^n$.
Assume that $F(t,0)=0$ for every $t \ge a$, $\left| F(t,x) \right| \le k(t)|x|$ for every $t \ge a$, $x \in \RR^n$ and some continuous function $k$ such that
$$
\int_a^\infty \frac{k(t)}{T(t,\a)}\, dt < \infty.
$$
Then the trivial solution $x\equiv 0$ of \eqref{eq:Cauchy} is uniformly stable.
\end{theorem}

\begin{proof}
Let us define
$$
N=\int_a^\infty \frac{k(t)}{T(t,\a)}\, dt < \infty.
$$
Theorem \ref{t:chain}, Proposition \ref{p:fundamental} and $k \ge 0$ give for every $t \ge t_0 \ge a$
\begin{equation}
\label{eq:boundedness11bis}
\begin{aligned}
& J_{ T,t_0 }^{\a } \big( k e ^{ -J_{ T,t_0 }^{\a }(k)}\big)(t) \, e ^{ J_{ T,t_0 }^{\a }(k)(t)}\\
& = J_{ T,t_0 }^{\a } \big( G_T^\a (- e ^{ -J_{ T,t_0 }^{\a }(k)}) \big)(t) \, e ^{ J_{ T,t_0 }^{\a }(k)(t) }
\\
& = \big( e ^{ -J_{ T,t_0 }^{\a }(k)(t_0)} - e ^{ -J_{ T,t_0 }^{\a }(k)(t)}\big) \, e ^{ J_{ T,t_0 }^{\a }(k)(t)}
\\
& = \big( 1 - e ^{ -J_{ T,t_0 }^{\a }(k)(t)}\big) \, e ^{ J_{ T,t_0 }^{\a }(k)(t)}
\\
& = e ^{ J_{ T,t_0 }^{\a }(k)(t)} - 1
\\
& \le e ^{ J_{ T,a }^{\a }(k)(t)} - 1
\\
& \le e^N -1  .
\end{aligned}
\end{equation}

Fix $\varepsilon >0$ and let $x=x(t)$ be the solution of the Cauchy problem \eqref{eq:Cauchy}, where ${ x }_{ 0 }$ satisfies
\begin{equation}
\label{eq:boundedness12}
\left| { x }_{ 0 } \right|
<\delta
= \delta(\varepsilon)
= e^{-N} \varepsilon  .
\end{equation}
Proposition \ref{p:equivalence} gives
$$
x(t)={ x }_{ 0 }+J_{ T,t_0 }^{\a }F(s,x(s))(t),
$$
for every $t \ge t_0$.
Thus, Theorem \ref{t:pos} gives
$$
\begin{aligned}
\left| x(t) \right|
& \le \left| { x }_{ 0 } \right| +\left|J_{ T,t_0 }^{\a }F(s,x(s))(t)  \right|\\
&\le \left| { x }_{ 0 } \right| + J_{ T,t_0 }^{\a }\left|F(s,x(s)) \right|(t)\\
&\le \left| { x }_{ 0 } \right| + J_{ T,t_0 }^{\a }(k\left|x \right|)(t) ,
\end{aligned}
$$
for every $t \ge t_0$.
Hence, Theorem \ref{t:gronwall}, \eqref{eq:boundedness11bis} and \eqref{eq:boundedness12} give
$$
\begin{aligned}
\left| x(t) \right|
& \le \left| x_0 \right| + \left| x_0 \right| J_{ T,t_0 }^{\a } \big( k e ^{ -J_{ T,t_0 }^{\a }(k)}\big)(t) \, e ^{ J_{ T,t_0 }^{\a }(k)(t)}
\\
& \le \left| x_0 \right| + \left| x_0 \right| (e^N-1)
= \left| x_0 \right| e^N
< \e ,
\end{aligned}
$$
for all $t\ge t_0$.
Since $\d$ does not depend on $t_0$, the trivial solution of \eqref{eq:Cauchy} is uniformly stable.
\end{proof}

We will present the basic results of the Stability Theory for the  Generalized Bernoulli Differential Equation:
\begin{equation}
\label{b:1}
{ G }_{ T }^{ \alpha  }y+p(x)y=q(x){ y }^{ n }, \qquad n\neq 0,1.
\end{equation}

First, we will pose the problem for a much more general system that \ref{b:1}:

\begin{equation}
\label{e:1}
G_{ T }^{ \alpha  }y(x)=f(x,y(x)),
\end{equation}

\begin{equation}
\label{i:1}
y(x_{0})=y_{0},
\end{equation}

where $f\in C(\mathbb{R}_{+}\times
\mathbb{R},\mathbb{R})$, $x_{0}>0$. It is further assumed that for $(x_{0},y_{0})\in $ int$(
\mathbb{R}_{+}\times \mathbb{R})$ the initial value problem (\ref{e:1})-(\ref{i:1}) has a solution $y(x)\in C^{\alpha
}(I)$ for all $x>x_{0}>0$. In addition, it is assumed that $f(x,0)=0$ for
all $x>x_{0}>0$.\\

Several types of stability can be discussed for the solutions of differential equations (of integer or fractional order) that describe dynamic systems. The most important type the one related to the stability of the solutions near a point of equilibrium. This can be discussed by Lyapunov's theory. In simple terms, if solutions that start near an equilibrium point ${ y }_{ e }$ stay close to ${ y }_{ e }$ for all $x$ then ${ y }_{ e }$  it's stable Lyapunov. More strongly, if ${ y }_{ e }$ is Lyapunov stable and all solutions starting near ${ y }_{ e }$  converge to ${ y }_{ e }$, then ${ y }_{ e }$ it is asymptotically stable.

The notion of exponential stability (for linear equations or systems) guarantees a minimum rate of decay, that is, an estimate of the speed with which the solutions converge.
The stability of a solution of  fractional diferential equations can be defined in exactly the same way. The following results form a certain theory of stability for equations of the type of (\ref{e:1}). Firstly, we will state the following lemma, which will be a basis for obtaining the desired stability results.

\begin{lemma}
Let $I$ be an interval with $[t_0, \infty) \subseteq I$ and $\a \in (0,1]$.
Let $p$ be a continuous function on $[t_0, \infty)$, and consider the linear equation
\begin{equation} \label{b:2}
{ G }_{ T }^{\alpha}y+p(t)y=0.
\end{equation}

$(1)$ If $\liminf_{t \to \infty} J_{T,t_0}^\a(p)(t) > - \infty$, then the trivial solution $y(t)\equiv 0$ is stable.

$(2)$ If $\lim_{t \to \infty} J_{T,t_0}^\a(p)(t) = \infty$, then $y(t)\equiv 0$ is asymptotically stable.
\end{lemma}

Now consider the following initial value problem
\begin{equation}
\label{b:3}
{ G }_{ T }^{ \alpha  }y+p(x)y=0,  y(x_{0})=y_{0},
\end{equation}
and show that the trivial solution $y(t)\equiv 0$ of (\ref{b:3}) is stable if and only if
\begin{equation}
\label{Jlim}
J_{ T,0 }^{ \alpha  }(p)(+\infty )=+\infty,
\end{equation}
then, by means of the properties of the integral operator, we can see that the general solution of (\ref{b:3}) is
\begin{equation}
y(x)={ y }_{ 0 }{ e }^{ -{ J }_{ { T,x }_{ 0 } }^{ \alpha  }p (x) },
\end{equation}
from which the statement \ref{Jlim} is easily obtained.\\

We are now in a position to establish the main results on the stability of equations of type (\ref{e:1}).
\begin{theorem}
\label{t:1}
Consider the differential equation

\begin{equation}
\label{b:4}
{ G }_{ T }^{ \alpha  }y+p(x)y=g(x,y), \quad x>0,
\end{equation}
where $p(x)$ is a positive and continuous function such that $0<p\le p(x)$, and $g$ is a continuous function with $g(x,0)\equiv 0$ for all $x$. If

\begin{equation}
\label{limG}
\lim_{ z\to 0} \sup_{x>0} \frac { \left| g(x,z) \right|  }{ z } =0,
\end{equation}
then the trivial solution $y(x)\equiv 0$ is an asymptotically stable solution.
\end{theorem}

\begin{proof}
Choose any $\varepsilon \in (0,p)$. By (\ref{limG}), there exists $\delta>0$ such that

\begin{equation*}
\sup_{x>0} \left| g(x,z) \right|
\le \varepsilon \left| z \right| ,\quad for\quad all\quad \left| z \right| \le \delta.
\end{equation*}

Let ${ y }_{ 0 }\in (0,\delta )$. Since no two solutions can intersect, we have $y(x) > 0$ (or $y(x) < 0$) for all $x>0$. Consider first that $y(x) > 0$,  taking $y({ x }_{ 0 })={ y }_{ 0 }$ we obtain from (\ref{b:4}) that $-p(x)y+g(x,y)\le -py+\varepsilon y<0$, i.e., the solution $y(x)$ is decreasing. Therefore, we have ${ G }_{ T }^{ \alpha  }y(t)\le -(p-\varepsilon )y$ but this is a particular case of (\ref{b:3}), so we have $0<y(x)\le { y }_{ 0 }{ e }^{ -{ J }_{ { T,x }_{ 0 } }^{ \alpha  }(p-\varepsilon )(x) }\rightarrow 0,\quad as\quad x\rightarrow +\infty$.
Similarly, we can prove the $y(x)\rightarrow 0$ as $x\rightarrow +\infty$ if ${ y }_{ 0 }\in (-\delta,0 )$. Hence the trivial solution $y(x)\equiv 0$ is asymptotically stable.
\end{proof}

The previous theorem, allows us to establish conditions to ensure the asymptotic stability of the Bernoulli Equation (\ref{b:1}). In this way we have the following

\begin{corollary}
Under assumption of Theorem \ref{t:1} on function $p(x)$, be $q(x)$ a bounded function, then the trivial solution $y(x)\equiv 0$ of (\ref{b:1}) is asymptotically stable.
\end{corollary}
The following result shows how the integral operator $J_{T,a}^\a$ can be applied in order to solve fractional differential equations.
\begin{theorem}
\label{t:linearfr}
Let $p$ and $q$ be continuous functions on an interval $I \subseteq \RR$, $a \in I$, $\a \in (0,1]$ and $n \in \RR \setminus \{1\}$, and consider
the Bernoulli fractional equation
\begin{equation}
\label{eq:linearfr1}
G_{ T }^{ \alpha  }y+p(t)y=q(t)y^n .
\end{equation}
Then the following statements hold.

$(1)$ For each $C \in \RR$, the function
\begin{equation}
\label{eq:linearfr2}
\begin{aligned}
y(t)= & [{ e }^{ -(1-n)J_{T,a}^\a (p)(t) }( (1-n)J_{T,a}^\a \big( q { e }^{ (1-n)J_{T,a}^\a(p) }\big)(t)\\
&+C ) ]^{1/(1-n)}
\end{aligned}
\end{equation}
is a solution of the Bernoulli fractional equation \eqref{eq:linearfr1}.

$(2)$ For each $t_0 \in I$ and $y_0 \in \RR$, let $y(t;t_0,y_0)$ be the function in \eqref{eq:linearfr2} with
$$
C=
y_0^{1-n}{ e }^{ (1-n)J_{T,a}^\a (p)(t_0) } - (1-n)J_{T,a}^\a \big( q { e }^{ (1-n)J_{T,a}^\a(p) }\big)(t_0).
$$
If $(y_0^{1-n})^{1/(1-n)} = y_0$ and $y(t;t_0,y_0)$ is defined in some right or left neighborhood $U \subseteq I$ of $t_0$,
then $y(t;t_0,y_0)$ is a solution of \eqref{eq:linearfr1} on $U$ satisfying the initial condition $y(t_0;t_0,y_0)=y_0$.

$(3)$ For each $t_0 \in I$ and $y_0 > 0$ there exists a unique solution of \eqref{eq:linearfr1}
satisfying the initial value $y(t_0)=y_0$, given by $y(t;t_0,y_0)$.

$(4)$ If $n\in(1,\infty) \cap \QQ$ and $1/(1-n)=r/s$ with $s$ an odd integer,
then for each $t_0 \in I$ and $y_0 \in \RR$ there exists a unique solution of \eqref{eq:linearfr1}
satisfying the initial value $y(t_0)=y_0$, given by $y(t;t_0,y_0)$.
\end{theorem}

\begin{proof}
For each $C \in \RR$, let us define the function
\begin{equation}
\label{eq:linearfr3}
\begin{aligned}
z(t)= & { e }^{ -(1-n)J_{T,a}^\a (p)(t) }( (1-n)J_{T,a}^\a \big( q { e }^{ (1-n)J_{T,a}^\a(p) }\big)(t)\\
& +C ).
\end{aligned}
\end{equation}
Therefore, $y=z^{1/(1-n)}$.
Theorems \ref{t:prop} and \ref{t:chain} and Proposition \ref{p:inverse} give
$$
\begin{aligned}
G_{ T }^{\alpha}z
& = (1-n) G_{ T }^{\alpha}\big(-J_{T,a}^\a (p)\big) (t) \, e^{ -(1-n)J_{T,a}^\a (p)(t)}\\
& \times \left( (1-n) J_{T,a}^\a \big( q { e }^{ (1-n)J_{T,a}^\a(p) }\big)(t) +C \right)
\\
&\quad + (1-n) e^{ -(1-n)J_{T,a}^\a (p)(t) }G_{ T }^{\alpha}\left( J_{T,a}^\a \big( q { e }^{ (1-n)J_{T,a}^\a(p) }\big) \right)(t)
\\
& = -(1-n) p (t) \, e^{ -(1-n)J_{T,a}^\a (p)(t) }\\
& \times \left( (1-n) J_{T,a}^\a \big( q { e }^{ (1-n)J_{T,a}^\a(p) }\big)(t) +C \right)
\\
& \quad + (1-n)e^{ -(1-n)J_{T,a}^\a (p)(t) } q(t) { e }^{ (1-n)J_{T,a}^\a(p)(t) }
\\
& = -(1-n)p(t) z + (1-n)q(t) .
\end{aligned}
$$
Hence, Theorem \ref{t:chain} gives
$$
\begin{aligned}
G_{ T }^{\alpha}y
& = \frac1{1-n} \, z^{n/(1-n)} G_{ T }^{\alpha}z
\\
& = \frac1{1-n} \, z^{n/(1-n)} \Big( -(1-n)p(t) z + (1-n)q(t) \Big)
\\
& = y^{n} \big( -p(t) y^{1-n} + q(t) \big)
\\
& = -p(t)y+q(t)y^n ,
\end{aligned}
$$
and $y$ is a solution of the Bernoulli fractional equation
\eqref{eq:linearfr1}.

\smallskip

If $(y_0^{1-n})^{1/(1-n)} = y_0$ and $y(t;t_0,y_0)$ is defined in some right or left neighborhood $U \subseteq I$ of $t_0$,
then it is clear that $y(t;t_0,y_0)$ is a solution of \eqref{eq:linearfr1} on $U$ satisfying the initial condition $y(t_0;t_0,y_0)=y_0$.

\smallskip

By Theorem \ref{t:comp}, the fractional differential equation \eqref{eq:linearfr1} is equivalent to
$$
y'=\frac{-1}{T(t,\alpha )}\,p(t)y+\frac1{T(t,\alpha )}\,q(t)y^n =: f(t,y).
$$
Since $y_0>0$, the function $f(t,y)$ is continuous in a neighborhood $V$ of $(t_0,y_0)$ and, also, it is Lipschitz in the second variable in $V$.
Therefore, Picard's Theorem gives that
there exists a unique solution of \eqref{eq:linearfr1}
satisfying the initial value $y(t_0)=y_0$.
Since $y_0>0$, we have
$(y_0^{1-n})^{1/(1-n)} = y_0$, and the second item of this theorem gives that this solution is $y(t;t_0,y_0)$.

\smallskip

Assume now that $n\in(1,\infty) \cap \QQ$ and $1/(1-n)=r/s$ with $s$ an odd integer.
For each $y_0\in \RR$, the function $f(t,y)$ is continuous in a neighborhood $V$ of $(t_0,y_0)$;
since $n>1$, $f$ is Lipschitz in the second variable in $V$.
Hence, Picard's Theorem gives that
there exists a unique solution of \eqref{eq:linearfr1}
with $y(t_0)=y_0$.
Since $1/(1-n)=r/s$ with $s$ an odd integer, we have
$(y_0^{1-n})^{1/(1-n)} = y_0$ for every $y_0\in \RR$, and the second item of this theorem gives that this solution is $y(t;t_0,y_0)$.
\end{proof}

\begin{remark}
The results obtained, relative to the Bernoulli Equation, generalize and complete those obtained in \cite{MHHO}, obtained using the conformable derivative of \cite{Khalil}.
\end{remark}

The generalized derivative can be considered as a good tool to solve certain types of problems, therefore, by having the Bernoulli differential equation under this approach, we have a greater possibility to study and solve certain type of problems, as we will see in the next examples.\\

Consider the functions $p(t)=q(t)=1$, the interval $I=[0.5, 2]$ and $n=2$, according with Equation \eqref{eq:linearfr1}  we obtain,
\begin{equation}
\label{eq1}
G_{T}^{\alpha}y+y=y^{2},
\end{equation}
and consider the function $T(t, \alpha)=t^{1- \alpha}$ and $C=-1$, then by Theorem \eqref{t:linearfr} we obtain the following solution,
\begin{equation}
\label{y1}
y(t)=\left( 1 -2e^{\frac{t^{\alpha}- (0.5)^{\alpha}}{\alpha}} \right)^{-1},
\end{equation}
the solutions associated for differents values of $\alpha$ are shown in Figure \eqref{fg5}.

\begin{figure}[ht]
\includegraphics[width=8cm, height=7cm]
{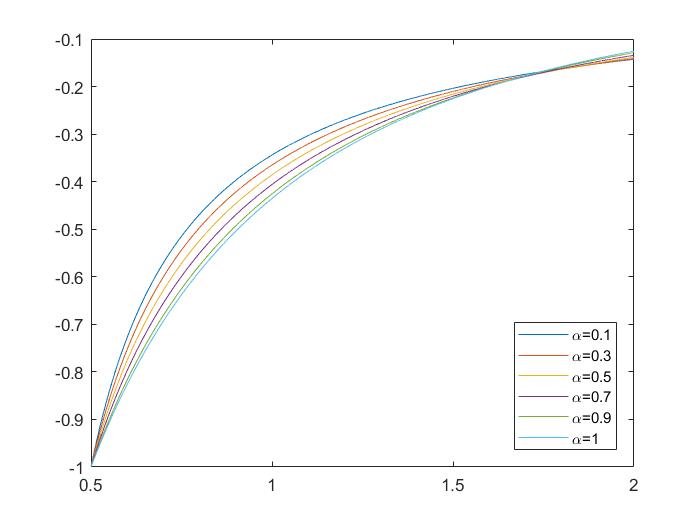}
\caption{Solution of Equation \ref{eq1} with $T(t, \alpha)=t^{1- \alpha}$ for $\alpha=0.1, 0.3, 0.5, 0.7, 0.9, 1$ .}
\label{fg5}
\end{figure}

As we can see, the solutions approach the curve associated with the value of $\alpha=1$, which corresponds to the case of the Bernoulli equation with derivative of order one, this is because the kernel $T(t, \alpha)=t^{1-\alpha}$ is associated with a conformable fractional derivative, however this does not happen in all cases, as we will show in the next example.\\

Now if we choose the functions $p(t)=q(t)=e^{t}$, the interval $I=[0.5, 6]$ and $n=2$, replacing this values in Equation (\ref{eq:linearfr1})  we obtain
\begin{equation}
\label{eq2}
G_{T}^{\alpha}y+e^{t}y=e^{t}y^{2},
\end{equation}
then, if we choose the kernel $T(t, \alpha)=e^{t-\alpha}$ and $C=1$, by Theorem \eqref{t:linearfr} we obtain,
\begin{equation}
\label{y2}
y(t)=\frac{1}{1-3e^{(t-0.5)e^{\alpha}}},
\end{equation}
in Figure \eqref{fg6} the solutions associated are shown for different values of $\alpha$.

\begin{figure}[H]
\includegraphics[width=8cm, height=7cm]
{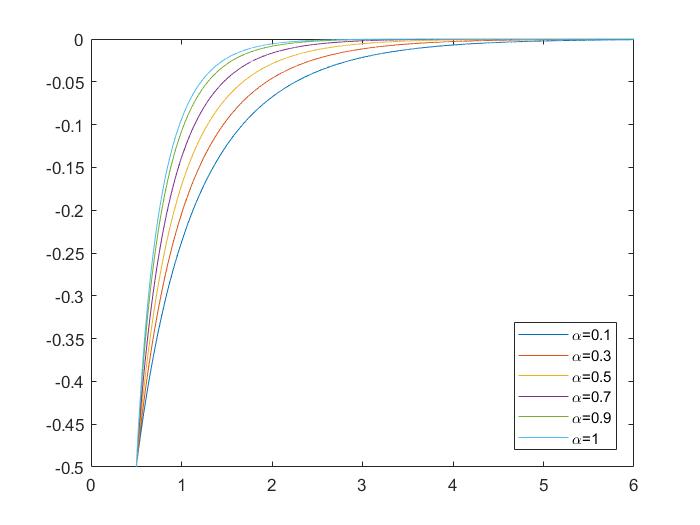}
\caption{Solution of Equation \eqref{eq2} with $T(t, \alpha)=e^{t- \alpha}$ for $\alpha=0.1, 0.3, 0.5, 0.7, 0.9, 1$ .}
\label{fg6}
\end{figure}

In this second example we can see that the solutions approximate in ascending order the solution associated with the case $\alpha=1$, in this case the kernel $T(t, \alpha)=e^{t-\alpha}$ corresponds to a non-conformable fractional derivative, however, we could obtain similar results to the first example, if we choose a suitable function $T(t, \alpha)$ in such a way that this function corresponds to a conformable derivative, as we can see, one of the advantages of considering the generalized derivative is precisely the freedom to choose the kernel in such a way that it is appropriate for the study and  solution of the problem under consideration, in addition to the fact that it generalizes the definitions of fractional derivative mentioned at the beginning of this work.

\subsection{A Finite Difference Method for the Generalized Bernoulli Differential Equation}

The equation,
\begin{equation}
\label{eq:FB}
\left\{\begin{array}{l}
G_{ T }^{ \alpha  }y+p(t)y=q(t)y^n, \, \, a\leq t \leq b,\\
y(a)=\beta,
\end{array}\right.
\end{equation}
according to Theorem \ref{t:comp} is equivalent to
$$
y'+\frac{1}{T(t,\alpha )}\,p(t)y=\frac1{T(t,\alpha )}\,q(t)y^n.
$$
Since we do not have any information at the right end of the interval, we cannot proceed as in the usual finite differences, where usually a tridiagonal matrix is constructed.

So, the strategy will be to transform the process to an iterative method based on the idea of the fixed point theory, then use the two-point finite difference formula to find the first value on the right using the initial condition as the initial value, then the three-point difference formula to generate the next three values and achieve a better order of convergence, and then finally move to the five-point difference formula with an order of convergence $\mathcal{O}(h^4)$ to finish the iteration.

It is worth mentioning that the method can be easily extended to other types of non-linearities, can also be divided into three independent methods that work with different degrees of convergence. By switching from forward to backward difference, you can decide whether you want to solve a nonlinear equation or not, although you have to choose between speed and convergence.

First, we select an integer $N>0$ and divide the interval $[a, b]$ into $(N+1)$ equal subintervals whose endpoints are the mesh points $x_i=a+i h$, for $i=0,1, \ldots, N+1$, where $h=(b-a) /(N+1)$.

Since the only information we do have is the $y_{0}=y(a)=\beta$ value given by the initial condition, we will start by using the two point forward difference formula to determine $y_{1}$, with truncation error $\mathcal{O}(h)$
$$
y^{\prime}\left(t_i\right)=\frac{y\left(t_i+h\right)-y\left(t_i\right)}{h}-\frac{h}{2} y^{\prime \prime}(\xi_0), \, \, (t_i<\xi_0<t_i+h) .
$$
By making $y\left(t_{i+1}\right)=y\left(t_i+h\right)$ the equation \eqref{eq:linearfr1} can be written as
$$
\frac{1}{ h}\left[y\left(t_{i+1}\right)-y\left(t_{i}\right)\right]+\frac{p(t_{i})}{T(t_{i},\alpha )}\,y(t_{i})=\frac{q(t_{i})}{T(t_{i},\alpha )}\,y(t_{i})^n,
$$
or,
\begin{equation} \label{O1}
y\left(t_{i+1}\right)=\frac{-hp(t_{i})}{T(t_{i},\alpha )}\,y(t_{i})+\frac{hq(t_{i})}{T(t_{i},\alpha )}\,y(t_{i})^n+y\left(t_{i}\right).
\end{equation}
Starting at $i=0$, the equation \eqref{O1} defines an iterative process that converges to the solution, this is the well known Euler method.

We can also use the finite difference backward formula (replacing $h$ by $-h$), then the iterative process is defined as follows

\begin{equation} \label{O2}
y\left(t_{i}\right)-y\left(t_{i-1}\right)+\frac{hp(t_{i})}{T(t_{i},\alpha )}\,y(t_{i})-\frac{hq(t_{i})}{T(t_{i},\alpha )}\,y(t_{i})^n=0.
\end{equation}
Where we must solve a nonlinear equation at each step of the iterative process to determine $y\left(t_{i}\right)$, the $fsolve$ command in Matlab R2022B can be used for this purpose, taking as an initial approximation the initial condition to start the process, that can then be updated with the previous point.

The equation \eqref{O2} improves the convergence with respect to the equation \eqref{O1} in exchange for consuming more time. So we must decide between a better convergence or a shorter execution time.

We have already found a solution to our problem, but recall that our goal was to use one of the above equations to find $y_1$. Using the three-point midpoint difference formula, with truncation error $\mathcal{O}(h^2)$, we can find $y_2$.
$$
y^{\prime}\left(t_i\right)=\frac{1}{2 h}\left[y\left(t_i+h\right)-y\left(t_i-h\right)\right]-\frac{h^2}{6} y^{(3)}\left(\xi_1\right), \, \, (t_i-h<\xi_1<t_i+h), 
$$
and since we know $y_0$ and $y_1$ and taking $i=1$ to start, the equation $\eqref{eq:FB}$ becomes,
$$
\frac{1}{2 h}\left[y\left(t_{i+1}\right)-y\left(t_{i-1}\right)\right]+\frac{p(t_{i})}{T(t_{i},\alpha )}\,y(t_{i})=\frac{q(t_{i})}{T(t_{i},\alpha )}\,y(t_{i})^n,
$$
or,
\begin{equation} \label{O3}
y\left(t_{i+1}\right)=\frac{-2hp(t_{i})}{T(t_{i},\alpha )}\,y(t_{i})+\frac{2hq(t_{i})}{T(t_{i},\alpha )}\,y(t_{i})^n+y\left(t_{i-1}\right).
\end{equation}
With \eqref{O3} we again obtain an iterative process that converges to the solution, with better results than previously found. Now let's try with the three-point backward difference formula
$$
y^{\prime}\left(t_i\right)=\frac{1}{2 h}\left[3y\left(t_i\right)-4y\left(t_{i-1}\right)+y\left({t_{i-2}}\right)\right]-\frac{h^2}{3} y^{(3)}\left(\xi_1\right), \, \, (t_i<\xi_1<t_i+2h), 
$$
the process is determined as follows
\begin{equation} \label{O4}
3y\left(t_i\right)-4y\left(t_{i-1}\right)+y\left(t_{i-2}\right)+\frac{hp(t_{i})}{T(t_{i},\alpha )}\,y(t_{i})-\frac{hq(t_{i})}{T(t_{i},\alpha )}\,y(t_{i})^n=0.
\end{equation}
Starting at $i=2$ we must again solve for $y\left(t_i\right)$ at each iteration and the same effects as above will be obtained.

Finally, let us use the initial condition $y_0$, take $y_1$ from the first iteration of \eqref{O1} or \eqref{O2} and $y_3$ and $y_4$ from \eqref{O3} or \eqref{O4} to develop a last method from the backward five-point finite difference formula with truncation error $\mathcal{O}(h^4)$. It is not recommended to use a higher order approximation scheme due to Runge's phenomenon, instead other techniques including for example irregular grids, Chebyshev polynomials or spectral methods can be used.

The backward five-point difference formula is determined by the formula
$$
\begin{aligned}
y^{\prime}\left(t_i\right)&=\frac{1}{12 h}\left[25y\left(t_i\right)-48y\left(t_i-h\right)+36y\left(t_i-2h\right)-16y\left(t_i-3h\right)+3y\left(t_i-4h\right)\right]\\
& -\frac{h^4}{5} y^{(3)}\left(\xi_2\right) \, \, \, (t_i<\xi_2<t_i+4h),
\end{aligned}
$$
we obtain then,
\begin{equation} \label{O5}
\begin{aligned}
25y\left(t_i\right)&-48y\left(t_{i-1}\right)+36y\left(t_{i-2}\right)-16y\left(t_{i-3}\right)+3y\left(t_{i-4}\right)+\frac{12hp(t_{i})}{T(t_{i},\alpha )}\,y(t_{i})\\
&-\frac{12hq(t_{i})}{T(t_{i},\alpha )}\,y(t_{i})^n=0.
\end{aligned}
\end{equation}
Which again must be solved using some method for nonlinear equations.\\


Let's solve again equation \eqref{eq1} given in example 1 using the finite difference method (FDM), Table \ref{Cmp1} and Figure \ref{fig} show the solution and error between the finite difference method  and the exact solution of equation \eqref{eq1} for $\alpha=0.5$, we have zoomed in on the Figure \ref{fig} to show the difference between the curves.

\begin{figure}[ht]
\includegraphics[width=10cm, height=7cm]
{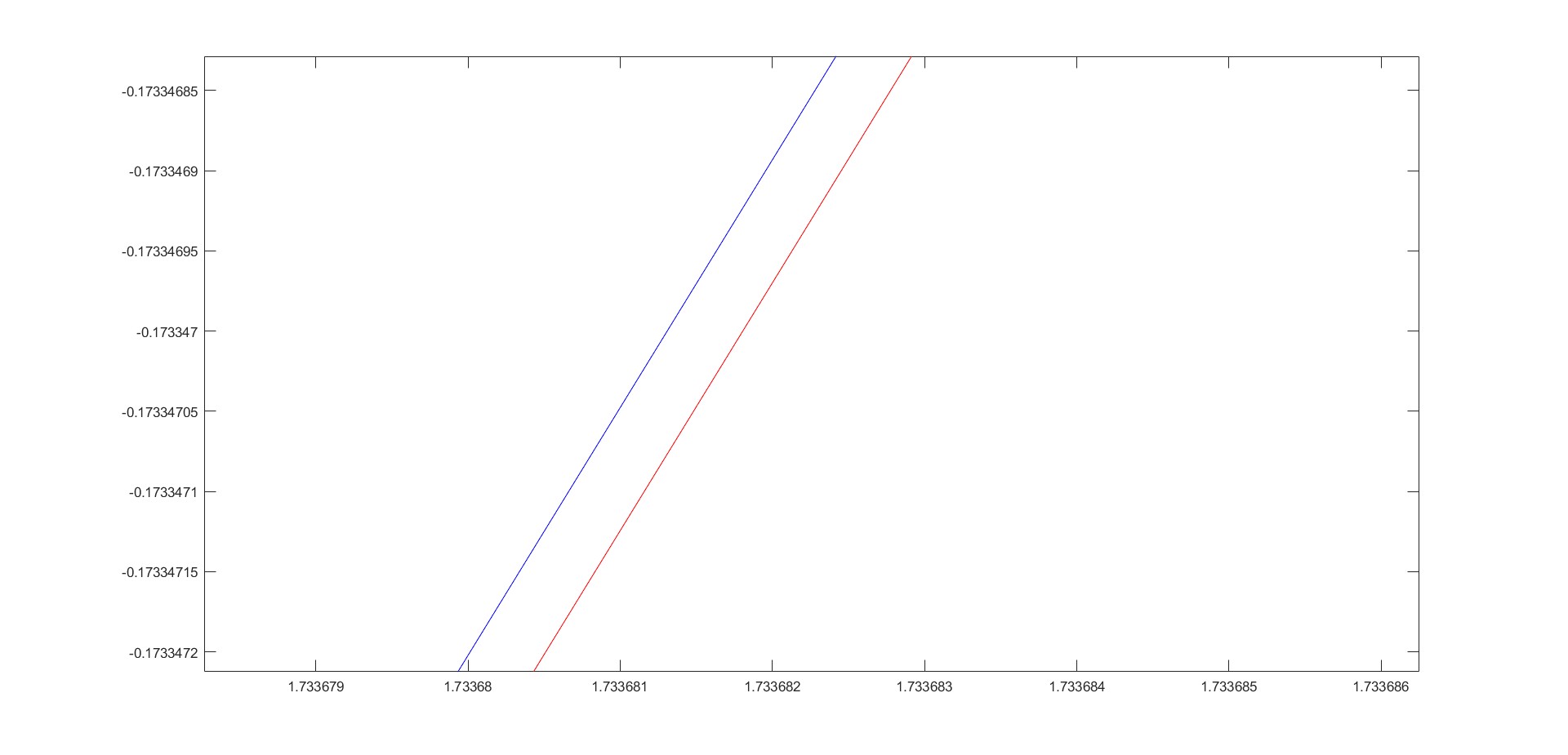}
\caption{Comparison between the FDM and the exact solution for (\ref{eq1}).}
\label{fig}
\end{figure}

\begin{table}[ht]
\begin{tabular}{|l|l|l|l|}
\hline Pts & Exact Sol & FDM & Error \\ \hline
0.5 & -1 & -1 & 0 \\ \hline
0.500299940011998 & -0.999152308016293 & -0.999152149990093 & 1.58026200125505e-07 \\ \hline
0.500599880023995 & -0.998305946878573 & -0.998305948338594 & 1.46002121503841e-09 \\ \hline
...... & ...... & ...... & ...... \\ \hline
1.99970005998800 & -0.138412976596618 & -0.138413036296161 & 5.96995426982439e-08 \\ \hline
2.00000000000000 & -0.138379560668148 & -0.138379620351763 & 5.96836149668878e-08 \\ \hline
\end{tabular}%

\caption{Comparison between the FDM and the exact solution for (\ref{eq1}) with $N=5000$ interval divisions.}
\label{Cmp1}
\end{table}

The above comparison shows that the proposed finite difference method is indeed feasible and that it delivers a solution quite close to the exact solution with considerably small error and low run time. In the following example we will show that the method can also be applied to more complicated problems, where it is sometimes difficult to obtain an exact solution.

There are cases in which the functions $p(t)$ and $q(t)$ make it difficult the use of the Theorem \eqref{t:linearfr} because complicated integrals must be solved, herein lies the usefulness of numerical methods.
If we choose the functions $p(t)=e^{-t^2}$, $q(t)=t$, $I=[-3,3]$ and $n=2$, we get the equation
\begin{equation}
\label{CE}
G_{T}^{\alpha}y+e^{-t^2}y=ty^{2}.
\end{equation}

Taking $T(t, \alpha)=e^{(1-\alpha)t}$ and $C=1$, we can find the solution to the previous equation through the FDM, and since we don't have an exact solution for comparison we may approximate the integrals in Theorem \eqref{t:linearfr} by a numerical integration method. This will allow us to validate again the proposed numerical method, as we do not have an explicit solution of the fractional equation. Table \ref{Cmp2} shows both solutions for $\alpha=0.5$, graphically, the difference between the two solution curves is hardly noticeable. Figure \ref{figBC} displays the solutions associated for differents values of $\alpha$ for the FDM.
\begin{table}[ht]
\begin{tabular}{|l|l|l|l|}
\hline Pts & Integ & FDM & abs(FDM-Integ) \\ \hline
-3 & 1 & 1 & 0 \\ \hline
-2.99880023995201 & 0.984132376988178 & 0.983900701320463 & 0.000231675667714537 \\ \hline
-2.99760047990402 & 0.968775560311577 & 0.968798608504813 & 2.30481932365079e-05 \\ \hline
...... & ...... & ...... & ...... \\ \hline
2.99880023995201 & 0.0115194694289867 & 0.0115187854523517 & 6.83976635032849e-07 \\ \hline
3.00000000000000 & 0.0115195756295232 & 0.0115188916402543 & 6.83989268873003e-07 \\ \hline
\end{tabular}
\caption{Comparison between FDM and aproximation of the integrals in the exact solution.}
\label{Cmp2}
\end{table}

\begin{figure}[ht]
\includegraphics[width=11cm, height=7cm]
{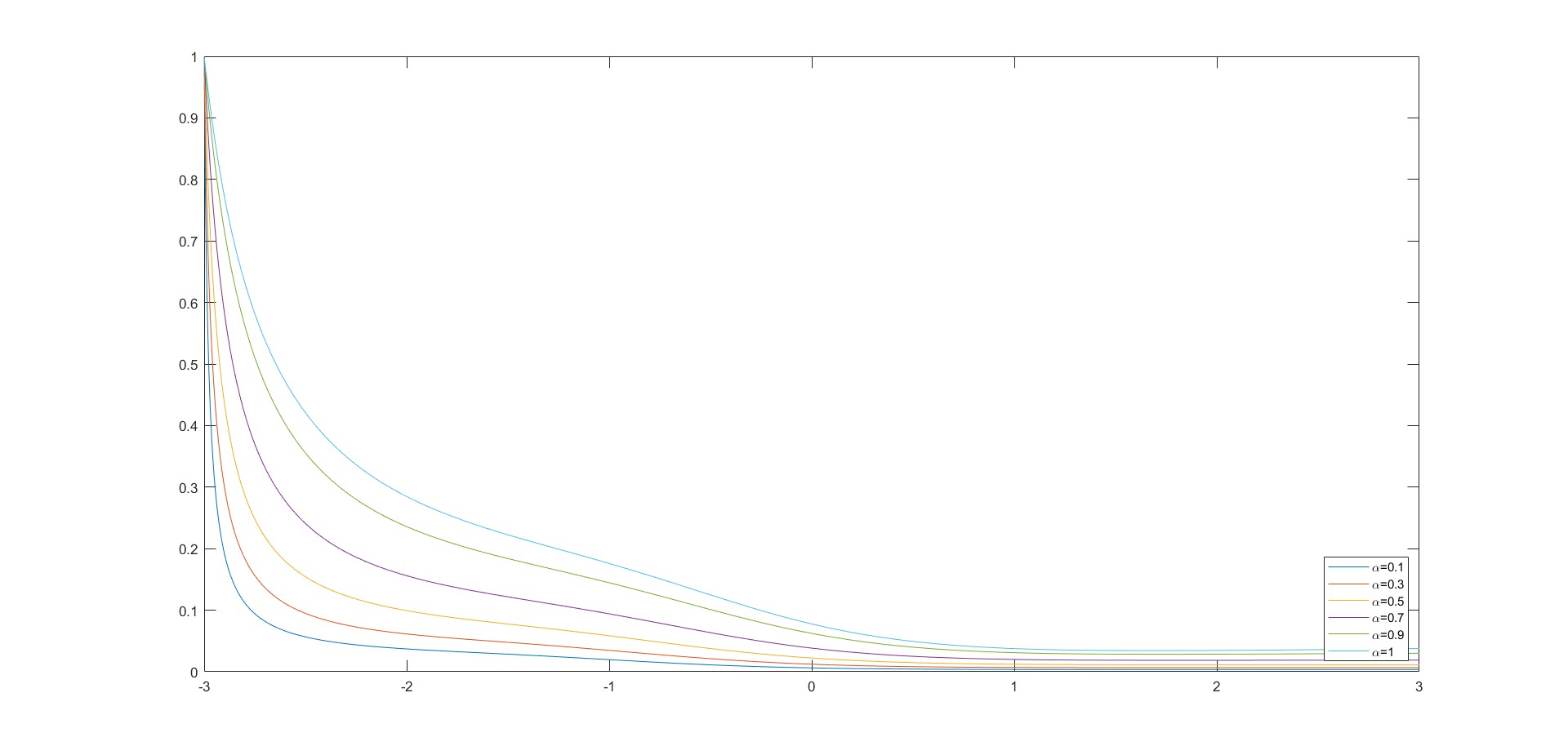}
\caption{Solution of Equation \eqref{CE} with $T(t, \alpha)=e^{(1-\alpha)t}$ for $\alpha=0.1, 0.3, 0.5, 0.7, 0.9, 1$.}
\label{figBC}
\end{figure}
In this example we can see that the solutions approximate in ascending order the case $\a = 1$ with the kernel $T(t, \alpha)=e^{(1-\alpha)t}$,  that corresponds to a conformable derivative.

We have been able to compare two different approaches, one in which the derivative is approximated by an appropriate difference quotient and another in which the integral has been approximated within the explicit solution. The approximation of the integral has been performed by numerical methods using Matlab R2022B, although it has proved to be an alternative to FDM, it has required much more time to deliver the solution, showing an increase of time as the complexity of the integrals and the number of nodes increases.

The examples shown above prove that the numerical proposed method (FDM) is a valid alternative when complicated integrals must be solved in the exact solution, which broadens the number of problems we can solve.

\section{Conclusions}
In this article we proposed and solved a generalization of the Bernoulli differential equation under the generalized derivative approach, we found solubility conditions and results about the qualitative behavior of the trivial solution. For this, a generalization of Gronwell's inequality was proved as well as its reciprocal and a particular case of this inequality. After that, we shown by means of examples how this  fractional derivative approach has some advantages over some other definitions, for example, this derivative  generalizes certain definitions of fractional derivative known in the literature and further, it allows us to choose the kernel of the derivative depending on the problem under consideration, so, we can solve different problems under different derivatives approaches by choosing a suitable kernel function $T(t, \alpha)$. We also proposed and tested the reliability of a finite difference method by means of examples, for the case in which the explicit solution involves complex integrals, comparing the solution obtained by FDM with the approximation of the integrals in the explicit solution of the fractional equation, thereby extending the number of problems we can solve.

\section*{Availability of data and material}
None

\section*{Competing interests}
The authors declare no competing  interests.

\section*{Funding}
None

\section*{Authors' contributions}
All the authors contributed equally to the work.

\section*{Acknowledgements}
None

\end{document}